\documentclass[11pt]{article}
\usepackage[margin=1in]{geometry} 
\usepackage{amssymb,amsfonts,amsmath,bbm,mathrsfs,stmaryrd,mathtools}
\usepackage{xcolor}
\usepackage{url}

\usepackage{extarrows}

\usepackage[shortlabels]{enumitem}
\usepackage{tensor}

\usepackage[T1]{fontenc}

\usepackage[colorlinks,
linkcolor=black!75!red,
citecolor=blue,
pdftitle={},
pdfproducer={pdfLaTeX},
pdfpagemode=None,
bookmarksopen=true,
bookmarksnumbered=true,
backref=page]{hyperref}

\usepackage{tikz}
\usetikzlibrary{arrows,calc,decorations.pathreplacing,decorations.markings,intersections,shapes.geometric,through,fit,shapes.symbols,positioning,decorations.pathmorphing}

\usepackage{braket}

\usepackage[amsmath,thmmarks,hyperref]{ntheorem}
\usepackage{cleveref}

\newcommand\nthalias[1]{\AddToHook{env/#1/begin}{\crefalias{lemma}{#1}}}

\nthalias{definition}
\nthalias{example}
\nthalias{examples}
\nthalias{remark}
\nthalias{remarks}
\nthalias{convention}
\nthalias{notation}
\nthalias{construction}
\nthalias{theoremN}
\nthalias{propositionN}
\nthalias{corollaryN}
\nthalias{lemma}
\nthalias{proposition}
\nthalias{corollary}
\nthalias{theorem}
\nthalias{conjecture}
\nthalias{question}
\nthalias{assumption}

\creflabelformat{enumi}{#2#1#3}

\crefname{section}{Section}{Sections}
\crefformat{section}{#2Section~#1#3} 
\Crefformat{section}{#2Section~#1#3} 

\crefname{subsection}{\S}{\S\S}
\AtBeginDocument{%
  \crefformat{subsection}{#2\S#1#3}%
  \Crefformat{subsection}{#2\S#1#3}%
}

\crefname{subsubsection}{\S}{\S\S}
\AtBeginDocument{%
  \crefformat{subsubsection}{#2\S#1#3}%
  \Crefformat{subsubsection}{#2\S#1#3}%
}

\theoremstyle{plain}

\newtheorem{lemma}{Lemma}[section]

\newtheorem{corollary}[lemma]{Corollary}
\newtheorem{theorem}[lemma]{Theorem}

\theoremstyle{plain}
\theoremnumbering{Alph}

\theoremstyle{plain}
\theorembodyfont{\upshape}
\theoremsymbol{\ensuremath{\blacklozenge}}

\newtheorem{definition}[lemma]{Definition}

\newtheorem{examples}[lemma]{Examples}
\newtheorem{remark}[lemma]{Remark}

\crefname{definition}{definition}{definitions}
\crefformat{definition}{#2definition~#1#3} 
\Crefformat{definition}{#2Definition~#1#3} 

\crefname{example}{example}{examples}
\crefformat{example}{#2example~#1#3} 
\Crefformat{example}{#2Example~#1#3} 

\crefname{examples}{example}{examples}
\crefformat{examples}{#2example~#1#3} 
\Crefformat{examples}{#2Example~#1#3} 

\crefname{remark}{remark}{remarks}
\crefformat{remark}{#2remark~#1#3} 
\Crefformat{remark}{#2Remark~#1#3} 

\crefname{remarks}{remark}{remarks}
\crefformat{remarks}{#2remark~#1#3} 
\Crefformat{remarks}{#2Remark~#1#3} 

\crefname{convention}{convention}{conventions}
\crefformat{convention}{#2convention~#1#3} 
\Crefformat{convention}{#2Convention~#1#3} 

\crefname{notation}{notation}{notations}
\crefformat{notation}{#2notation~#1#3} 
\Crefformat{notation}{#2Notation~#1#3} 

\crefname{table}{table}{tables}
\crefformat{table}{#2table~#1#3} 
\Crefformat{table}{#2Table~#1#3}

\crefname{lemma}{lemma}{lemmas}
\crefformat{lemma}{#2lemma~#1#3} 
\Crefformat{lemma}{#2Lemma~#1#3} 

\crefname{proposition}{proposition}{propositions}
\crefformat{proposition}{#2proposition~#1#3} 
\Crefformat{proposition}{#2Proposition~#1#3} 

\crefname{corollary}{corollary}{corollaries}
\crefformat{corollary}{#2corollary~#1#3} 
\Crefformat{corollary}{#2Corollary~#1#3} 

\crefname{theorem}{theorem}{theorems}
\crefformat{theorem}{#2theorem~#1#3} 
\Crefformat{theorem}{#2Theorem~#1#3} 

\crefname{enumi}{}{}
\crefformat{enumi}{#2#1#3}
\Crefformat{enumi}{#2#1#3}

\crefname{assumption}{assumption}{Assumptions}
\crefformat{assumption}{#2assumption~#1#3} 
\Crefformat{assumption}{#2Assumption~#1#3} 

\crefname{construction}{construction}{Constructions}
\crefformat{construction}{#2construction~#1#3} 
\Crefformat{construction}{#2Construction~#1#3} 

\crefname{equation}{}{}
\crefformat{equation}{(#2#1#3)} 
\Crefformat{equation}{(#2#1#3)}

\numberwithin{equation}{section}

\theoremstyle{nonumberplain}
\theoremsymbol{\ensuremath{\blacksquare}}

\newtheorem{proof}{Proof}
\newcommand\pf[1]{\newtheorem{#1}{Proof of \Cref{#1}}}

\newcommand\bC{{\mathbb C}}

\newcommand\bG{{\mathbb G}}

\newcommand\bZ{{\mathbb Z}}

\newcommand\cL{{\mathcal L}}

\newcommand\cP{{\mathcal P}}

\DeclareMathOperator{\id}{id}

\DeclareMathOperator{\Irr}{Irr}

\newcommand{\cat}[1]{\textsc{#1}}

\newcommand{\qedhere}{\mbox{}\hfill\ensuremath{\blacksquare}}

\title{Fixed-point permanence under actions by finite quantum groups}
\author{Alexandru Chirvasitu}

\begin{document}

\date{}

\newcommand{\Addresses}{{
  \bigskip
  \footnotesize

  \textsc{Department of Mathematics, University at Buffalo}
  \par\nopagebreak
  \textsc{Buffalo, NY 14260-2900, USA}  
  \par\nopagebreak
  \textit{E-mail address}: \texttt{achirvas@buffalo.edu}


}}

\maketitle

\begin{abstract}
  Given an action by a finite quantum group $\mathbb{G}$ on a von Neumann algebra $M$, we prove that a number of familiar $W^*$ properties are equivalent for $M$ and the fixed-point algebra $M^{\mathbb{G}}$ (i.e. hold or not simultaneously for the two algebras); these include being hyperfinite, atomic, diffuse and of type $I$, $II$ or $III$. Moreover, in all cases the canonical central projections of $M$ and $M^{\mathbb{G}}$ cutting out the summand with the respective property coincide. The result generalizes its classical-$\mathbb{G}$ analogue due to Jones-Takesaki. 
\end{abstract}


\vspace{.5cm}

\noindent{MSC 2020: 20G42; 46L67; 46L10; 16W22; 46L52
  
  
}


\section*{Introduction}

Recall \cite[Proposition 2.1.4]{jones_fin-act-hyp} that for any action of a finite group $\bG$ on a von Neumann algebra $M$, the fixed-point subalgebra $M^{\bG}$ is or is not of type I, II, III or finite, along with $M$ (i.e. these properties are equivalent for the two von Neumann algebras). The present note is aimed primarily at extending that result to its analogue for actions of finite {\it quantum} groups $\bG$ in the sense of \cite[Definition 1.1.1]{NeTu13}: unital $C^*$-algebras $A:=C(\bG)$ equipped with a unital $C^*$ morphism
\begin{equation*}
  A
  \xrightarrow{\quad}
  A\otimes A
  \quad
  \left(\text{minimal $C^*$ tensor product \cite[Definition IV.4.8]{tak1}}\right),
\end{equation*}
coassociative in the obvious sense and such that
\begin{equation*}
  \overline{(A\otimes \bC)\cdot \Delta A}^{\|\cdot\|}
  =
  A\otimes A
  =
  \overline{(\bC\otimes A)\cdot \Delta A}^{\|\cdot\|}
\end{equation*}
(products of sets meaning spans of products). $\bG$ is finite when $C(\bG)$ is finite-dimensional, and actions on von Neumann algebras are as in \cite[Definition 2.2]{wan-erg}: writing $L^{\infty}(\bG)$ for the von Neumann closure of $C(\bG)$ in the GNS representation attached to the \emph{Haar state} \cite[Theorem 1.2.1]{NeTu13} $(A,h)$, an action of $\bG$ on $M$ is a $W^*$ morphism
\begin{equation*}
  M
  \xrightarrow{\quad\rho\quad}
  M\otimes L^{\infty}(\bG)
  \quad
  \left(\text{\emph{spatial} $W^*$ tensor product \cite[Definition IV.5.1]{tak1}}\right)
\end{equation*}
such that
\begin{itemize}[wide]
\item $(\id\otimes\Delta)\circ \rho = (\rho\otimes\id)\circ \rho$;

\item and
  \begin{equation*}
    \overline{(\bC\otimes L^{\infty}(\bG))\cdot \rho M}^{W^*}
    =
    M\otimes L^{\infty}(\bG).
  \end{equation*}
\end{itemize}

That in place, the quantum analogue of the aforementioned \cite[Proposition 2.1.4]{jones_fin-act-hyp} reads:

\begin{theorem}\label{th:finqg-sametype}
  Let $M$ be a $W^*$-algebra equipped with an action by a finite quantum group $\bG$. $M$ and the fixed-point subalgebra $M^{\bG}$ then have the same canonical finite hyperfinite, atomic, diffuse and type-$\tau$ central projections for $\tau\in\{I,\ II,\ III\}$. 
\end{theorem}

References for the various properties mentioned in the statement are provided in \hyperref[exs:compidealprops]{Examples~\ref*{exs:compidealprops}}, and we refer to \cite[\S V.1]{tak1} for type-decomposition theory. 

\Cref{th:finqg-sametype} follows fairly easily from work done in \cite{dcy} on \emph{finite-index expectations} attached to quantum actions. A brief reminder on the various notions of finiteness introduced in \cite[D\'efinitions 3.6]{bdh} for {\it conditional expectations} $M\xrightarrow{E}N$ onto $W^*$-subalgebras (i.e. \cite[\S II.6.10]{blk} $\sigma$-weakly continuous norm-1 idempotents) will help provide some context.
\begin{itemize}[wide]
\item $E$ is {\it of finite index} if $E-\lambda$ is {\it completely positive} \cite[\S II.6.9]{blk} for some $\lambda>0$:
  \begin{equation}\label{eq:ppexp}
    E(x^*x)\ge \lambda x^*x
    ,\quad \forall x\in M\otimes M_n    
    \quad\text{for some fixed $\lambda>0$ and arbitrary $n\in \bZ_{>0}$};
  \end{equation}
  Equivalently \cite[Theorem 1]{fk_fin-ind}, it is enough to require \Cref{eq:ppexp} for $n=1$ alone, i.e. the positivity of $E-{\lambda}$ entails its complete positivity. In the language of \cite[D\'efinitions 3.6]{bdh}, being of finite index is equivalent to being {\it weakly} of finite index. 
  
\item $E$ is {\it strongly of finite index} if furthermore $M$, regarded as an {\it $M$-$N$-correspondence} \cite[D\'efinition 2.1(i)]{bdh} via $E$, has a finite {\it orthonormal basis} (\cite[D\'efinitions 1.6]{bdh}, \cite[\S 1.1]{dh-1} and \cite[Theorem 3.12]{pasch_wast-mod}; {\it (complete) quasi-orthonormal system} in \cite[Definitions 4.8 and 4.10 and Theorem 4.11]{zbMATH01544901}).
\end{itemize}
We will also substitute {\it Pimsner-Popa} or {\it PP} for the phrase {\it of finite index}, following \cite[\S 2.1]{pw_vna-act}.

The proof of \Cref{th:finqg-sametype} follows a small detour meant to abstract away from the specifics of type decompositions and the like, proving that PP expectations preserve certain classes of von-Neumann-algebra properties of which being of type $\tau\in\{I,\ II,\ III\}$, or diffuse, or hyperfinite are examples.

\begin{definition}\label{def:propp}
  Consider a $W^*$-algebra property $\cP$, transferred to projections $p\in N$ by saying that $p$ has (or is) $\cP$ if $pNp$ does (or is).
  \begin{enumerate}[(1),wide]
  \item $\cP$ is {\it ideal} if for every von Neumann algebra $N$
    \begin{equation}\label{eq:phasp}
      \cat{Proj}_{\cP}(N):=
      \left\{\text{$\cP$-projections }p\in N\right\}
      \subseteq
      \cat{Proj}(N):=
      \left\{\text{all $N$-projections}\right\}
    \end{equation}
    is a (lattice-theoretic \cite[Definition post Example 2.15]{bly_latt}) {\it ideal}:
    \begin{itemize}
    \item closed under taking suprema;
    \item and a {\it down-set} \cite[Definition post Example 1.13]{bly_latt}:
      \begin{equation*}
        p\text{ has }\cP\text{ and }q\le p
        \xRightarrow{\quad}
        q\text{ has }\cP.
      \end{equation*}
    \end{itemize}
  \item $\cP$ is {\it completely} ideal if it is ideal and \Cref{eq:phasp} is furthermore {\it complete}, i.e. closed under {\it arbitrary} suprema.

  \item $\cP$ is instead (only) {\it $Z$-completely (or centrally completely)} ideal if it is ideal and closed under {\it centrally orthogonal} \cite[pre Theorem V.1.8]{tak1} sums (i.e. sums of projections whose respective {\it central supports} \cite[pre Corollary IV.5.6]{tak1} are orthogonal).
  \end{enumerate}
\end{definition}

\Cref{th:ppcomplid} below shows that completely ideal properties travel well along PP expectations, \Cref{cor:typres} specializes this to the properties listed in \Cref{th:finqg-sametype}, and the latter follows. A version of \Cref{th:ppcomplid} applicable to \emph{strongly} PP expectations and $Z$-completely ideal properties is proven in \Cref{th:strongpp}, and might be of some independent interest given the context.

\subsection*{Acknowledgments}

I am grateful for instructive comments from K. De Commer, A. Freslon, P. So{\l}tan, M. Wasilewski and M. Yamashita.


\section{Permanence under sufficiently non-degenerate expectations}\label{se:ppexp}


Recall the Introduction's finiteness expectation conditions. 

\begin{remark}\label{re:fgenough}
  One simple observation that it will be convenient to take for granted now and then is that being strongly PP is equivalent to being PP and $M$ being finitely generated as a right $N$-module in the purely algebraic sense that
  \begin{equation*}
    M=x_1N+\cdots+x_sN
    \quad\text{for some}\quad
    x_i\in M. 
  \end{equation*}
  One direction is obvious; as for the other ($\Leftarrow$), note that in general, an $N$-$W^*$-module $X$ finitely generated as plain $N$-module has a finite orthonormal basis. One can induct on the number $s$ of generators:
  \begin{itemize}[wide]
  \item $p_1:=\braket{x_1\mid x_1}\in N$ can be assumed a projection after polar-decomposing the original $x_1$ \cite[Proposition 3.11]{pasch_wast-mod};
    
  \item this then gives a projection
    \begin{equation*}
      X\ni x
      \xmapsto{\quad Q:=\ket{x_1}\bra{x_1}\quad}
      x_1\braket{x_1\mid x}
      \in X
    \end{equation*}
    in $\cL_N(X)$;
  \item and the summand $(1-Q)X\le X$ complementary to $QX$ is generated by the (images of the) $s-1$ elements $x_j$, $2\le j\le s$, propelling the induction. 
  \end{itemize}
\end{remark}

A few simple preliminary remarks, useful enough to set out explicitly, but left as an exercise (in abstracting away from, say, \cite[Theorem V.1.19]{tak1}): 

\begin{lemma}\label{le:pp'}
  Let $N$ be a von Neumann algebra and $\cP$ a completely ideal $W^*$-algebra property.
  \begin{enumerate}[(1),wide]
  \item\label{item:le:pp':maxiscent} There is a largest $\cP$-projection $z_{N,\cP}\in N$, automatically central.

  \item\label{item:le:pp':sneg} The {\it strong negation} $\cP'$ of $\cP$ defined by
    \begin{equation*}
      N\text{ is }\cP'
      \xLeftrightarrow{\quad}
      \text{it has no non-zero $\cP$-projections}
    \end{equation*}
    is also completely ideal, and $z_{N,\cP'}=1-z_{N,\cP}$.  

  \item\label{item:le:pp':z} If $\cP$ is only $Z$-completely ideal, then there still is a largest {\it central} $\cP$-projection $z=z_{N,\cP}\in N$, and $1-z$ is the largest $\cP_Z'$-projection for the {\it $Z$-strong negation}
    \begin{equation*}
      N\text{ is }\cP_Z'
      \xLeftrightarrow{\quad}
      \text{it has no non-zero central $\cP$-projections}
    \end{equation*}
    of $\cP$.  \qedhere
  \end{enumerate}  
\end{lemma}

\begin{examples}\label{exs:compidealprops}
  \begin{enumerate}[(1), wide]
  \item\label{item:exs:compidealprops:tau} The property of being type $\tau\in\{I,\ II,\ III\}$ is completely ideal, as follows from downward closure \cite[Proposition 11 and Corollary 4 to Proposition 13 of \S I.6.8; Proposition 4 of \S I.8.3]{dixw} and the central decomposition \cite[Theorem V.1.19]{tak1} $1=z_I+z_{II}+z_{III}$.

  \item\label{item:exs:compidealprops:atomic} Recall (\cite[\S IV.2.2.1]{blk}, \cite[\S 10.21]{strat}) that a $W^*$-algebra is {\it atomic} if every non-zero projection dominates a minimal non-zero projection. Atomicity is completely ideal, with closure under suprema following from the general discussion on the relative position of two projections in \cite[discussion preceding Theorem V.1.41]{tak1}.

    The strong negation of atomicity is {\it diffuseness} (\cite[\S 29.2]{strat}, \cite[Theorem III.4.8.8]{blk}); it too is completely ideal, either as a direct simple exercise or by \Cref{le:pp'}\Cref{item:le:pp':sneg}.

  \item\label{item:exs:compidealprops:afd} Another completely ideal property is that of being {\it AFD} (short for {\it approximately finite-dimensional} \cite[Definition XIV.2.3]{tak3} and an alternative for {\it hyperfinite}, following \cite[p.175]{zbMATH03532021}). It means that arbitrary finite subsets of the von Neumann algebra in question ($N$, say) are arbitrarily approximable in the {\it $\sigma$-strong$^*$ topology} \cite[\S I.3.1.6]{blk} by finite-dimensional von Neumann subalgebras, and is equivalent (\cite[Theorem 2 and Corollary 5]{ell_afd-2}, \cite[\S 6.4]{zbMATH03940006}) to the {\it injectivity} \cite[\S IV.2.2]{blk} of $N$. 

    As to the fact that AFD-ness is completely ideal:
    \begin{itemize}[wide]
    \item The characterization \cite[Proposition IV.2.1.4]{blk} of injectivity in terms of expectations makes closure under projection cutting obvious.
    \item For AFD projections $p,q\in N$ the supremum $p\vee q$ is again AFD by the already-mentioned relative-position analysis of \cite[discussion preceding Theorem V.1.41]{tak1}.

    \item And closure under filtered inclusions of AFD subalgebras is self-evident from the very definition of approximate finite-dimensionality, hence closure under {\it arbitrary} suprema. 
    \end{itemize}
    
  \item\label{item:exs:compidealprops:fin} One example of a commonly-discussed property that of course is {\it not} completely ideal is finiteness: for infinite-dimensional $H$, the identity of $\cL(H)$ is the supremum of finite-rank (hence finite) projections. Finiteness {\it is} ideal though \cite[Theorem V.1.37]{tak1}, and in fact \cite[Lemma V.1.18]{tak1} $Z$-completely so.

  \item\label{item:exs:compidealprops:propinfin} {\it Proper infinitude} \cite[Definition V.1.15]{tak1}, on the other hand, is not ideal (let alone completely or $Z$-completely so) because it is not closed under projection-cutting: the same example $N:=\cL(H)$ above is properly infinite (i.e. its non-zero projections are all infinite), but $pNp$ is finite for finite-rank $p$.

    Note in passing that proper infinitude is precisely the $Z$-strong negation of finiteness in the sense of \Cref{le:pp'}\Cref{item:le:pp':z}.
  \end{enumerate}
\end{examples}

\begin{remark}
  \Cref{le:pp'}\Cref{item:le:pp':z} is intended precisely as stated: for $Z$-completely ideal $\cP$ $z_{N,\cP}$ is by definition the largest {\it central} $\cP$-projection in $N$, while its complement $z_{N,\cP_Z'}=1-z_{N,\cP}$ is the largest $\cP_Z'$-projection period, central or not.

  For illustration, consider Examples \ref{exs:compidealprops}\ref{item:exs:compidealprops:fin} and \ref{exs:compidealprops}\ref{item:exs:compidealprops:propinfin} above: in the central decomposition
  \begin{equation*}
    N\ni 1=z+z':=z_{\cat{finite}}+z_{\cat{properly infinite}}
  \end{equation*}
  of \cite[Theorem V.1.19]{tak1} the second central summand $z'$ dominates every properly infinite projection in $N$.
\end{remark}

\begin{theorem}\label{th:ppcomplid}
  Let $M\xrightarrow{E}N$ be a PP conditional expectation and $\cP$ a completely ideal property that descends along such expectations. The maximal central projections $z_{N,\cP}$ and $z_{M,\cP}$ then coincide. 
\end{theorem}
\begin{proof}
  Under the PP assumption we also have \cite[Th\'eor\`eme 3.5, Lemme 2.15 and \S 3.10]{bdh} a PP expectation $\cL_N(X_E)\xrightarrow{E_1}M$ from the von Neumann algebra \cite[Proposition 3.10]{pasch_wast-mod} of {\it adjointable} continuous $N$-module endomorphisms of the $M$-$N$-correspondence (hence also a right $N$-$W^*$-module) attached \cite[Proposition 2.8]{bdh} to the expectation $E$.

  Note furthermore that $X_E$ is {\it non-degenerate} \cite[\S 1.1]{dh-1} as an $N$-$W^*$-module: the ideal of $N$ generated by $\braket{x\mid y}$, $x,y\in X_E$ is (dense in) $N$. 
 
  \begin{enumerate}[(I),wide=0pt]
  \item\label{item:pr:ppcomplid:I}{\bf: $\cP$ or its absence are equivalent for $N$ and $\cL_N(X_E)$.} This is true of $N$ and $\cL_N(X)$ for {\it any} non-degenerate $N$-$W^*$-module $X$: by the general structure \cite[Theorem 3.12]{pasch_wast-mod} we have
    \begin{equation*}
      \cL_N(X)\cong p \left(N\otimes\cL(\ell^{2}(S))\right)p
      ,\quad
      S\text{ a set and }
      p\in N\otimes\cL(\ell^{2}(S))\text{ a projection}.
    \end{equation*}
    Indeed, if $(x_s)_{s\in S}$ is an orthonormal basis for $X$ (what \cite[Theorem 3.12]{pasch_wast-mod} provides) with corresponding projections $p_s:=\braket{x_s\mid x_s}\in N$, then
    \begin{equation}\label{eq:xopluss}
      X\cong \bigoplus_{s\in S} p_s N
      \quad\text{and}\quad
      \cL_N(X) = p \left(N\otimes\cL(\ell^{2}(S))\right)p
      \quad\text{with}\quad
      p:=\mathrm{diag}\left(p_s,\ s\in S\right).
    \end{equation}    
    The operations of projection-cutting (i.e. $\bullet\mapsto p\bullet p$) and tensoring with a type-$I$ factor both preserve completely ideal properties, transporting $\cP$ from $N$ to $\cL_N(X)$. Conversely, if the latter has property $\cP$ then so do its corners $p_sNp_s$ and non-degeneracy means that $\bigvee_s p_s=1\in N$, hence the claim.

  \item\label{item:pr:ppcomplid:II}{\bf: $\cP$ lifts from $N$ to $M$.} For it lifts from $N$ to $\cL_N(X_E)$ by \Cref{item:pr:ppcomplid:I}, and then descends to $M$ along $E_1$ by hypothesis.

  \item\label{item:pr:ppcomplid:III}{\bf: The strong negation $\cP'$ lifts from $N$ to $M$.} Consider the central decomposition $1=z_{M,\cP}+z_{M,\cP'}$ of \Cref{le:pp'}\Cref{item:le:pp':sneg}. We have \cite[\S 3.9]{bdh} a PP expectation
    \begin{equation*}
      zM
      \xrightarrow{\quad E_{z'}\quad}
      zN
      ,\quad
      z:=z_{M,\cP},
    \end{equation*}
    and the $\cP'$ property for $N$ (hence also for $zN$) entails it for $zM$ by \Cref{item:pr:ppcomplid:II} applied to $\cP'$ (completely ideal again by \Cref{le:pp'}\Cref{item:le:pp':sneg}). $z=z_{M,\cP'}$ thus vanishes, i.e. the target rephrased.

  \item\label{item:pr:ppcomplid:IV}{\bf: Conclusion.} Applying \Cref{item:pr:ppcomplid:II} and \Cref{item:pr:ppcomplid:III} to the expectations
    \begin{equation*}
      \begin{aligned}
        z_{N,\cP}Mz_{N,\cP} &\xrightarrow{\quad E\text{ restricted}\quad} z_{N,\cP}N
                              \quad\text{and}\\
        z_{N,\cP'}Mz_{N,\cP'} &\xrightarrow{\quad E\text{ restricted}\quad} z_{N,\cP'}N
      \end{aligned}
    \end{equation*}
    respectively, we obtain an orthogonal decomposition $1=z_{N,\cP}+z_{N,\cP'}$ into a $\cP$ and a $\cP'$ summand. Now, $z:=z_{N,\cP}$ and $z':=z_{N,\cP'}$ cannot have non-zero sub-projections
    \begin{equation*}
      p\le z
      \quad\text{and}\quad
      p'\le z'
      \quad\text{equivalent in }M,
    \end{equation*}
    for then $pMp\cong p'Mp'$ would be both $\cP$ and non-$\cP$. It follows \cite[Lemma V.1.7]{tak1} that $z$ and $z'$ are centrally orthogonal in $M$. Being themselves orthogonal and complementary, $z$ and $z'$ must also be central in $M$, so that $1=z+z'$ is the canonical central decomposition of $M$ attached to property $\cP$.
  \end{enumerate}
\end{proof}

\cite[\S 2.1]{pw_vna-act} (essentially) claims, and \cite[Propositions 2.1 and 2.2]{fk_fin-ind} imply \Cref{cor:typres} (give or take, minus the AFD property) for PP expectations.

\begin{corollary}\label{cor:typres}
  If there is a PP conditional expectation $M\xrightarrow{E}N$, then the maximal central projections
  \begin{enumerate}[(a)]
  \item\label{item:cor:typres:tau} $z_{\tau}$ of respective types $\tau\in\{I,\ II,\ III\}$
  \item\label{item:cor:typres:at} $z_a$ (atomic)
  \item\label{item:cor:typres:dif} $z_d$ (diffuse)
  \item\label{item:cor:typres:afd} and $z_{AFD}$ (AFD)
  \end{enumerate}
  for $M$ and $N$ coincide. 
\end{corollary}
\begin{proof}
  Given that diffuseness and type $III$ are the respective strong negations of atomicity and {\it semifiniteness} \cite[\S\S III.1.4.2 and III.1.4.7]{blk}, the claims are all direct consequences of \Cref{th:ppcomplid} once we recall that
  \begin{itemize}[wide]
  \item semifiniteness \cite[Corollary III.2.5.25(i)]{blk}
  \item being of type $I$ \cite[Theorem III.2.5.26]{blk} (or {\it discrete} \cite[\S III.1.4.4]{blk})
  \item atomicity \cite[Theorem IV.2.2.3]{blk}
  \item and AFD-ness (obviously by \cite[Proposition IV.2.1.4]{blk} and the equivalence \cite[Theorem 2 and Corollary 5]{ell_afd-2} AFD $\Leftrightarrow$ injective)
  \end{itemize}
  all descend along {\it arbitrary} (not just PP) conditional expectations.
\end{proof}

\cite[Th\'eor\`eme 2.2(1)]{zbMATH00059991} shows that finiteness lifts along strongly PP expectations. \cite[Property 1.1.2(iii)]{popa_clsf-subf} makes the parallel claim under only the weak PP assumption, while \cite[\S 2.1]{pw} amplifies that claim with the assertion that the finite / properly infinite central decompositions of $M$ and $N$ coincide in the presence of a weakly PP expectation. The weaker version assuming {\it strong} PP-ness (strengthening \cite[Th\'eor\`eme 2.2(1)]{zbMATH00059991}) can be recovered by specializing \Cref{th:strongpp} below to $\cP$ = finiteness (per \Cref{exs:compidealprops}\Cref{item:exs:compidealprops:fin}). 

\begin{theorem}\label{th:strongpp}
  Let $M\xrightarrow{E}N$ be a strongly PP conditional expectation and $\cP$ a $Z$-completely ideal property that descends along such expectations. The central decompositions
  \begin{equation*}
    z_{N,\cP} + z_{N,\cP_Z'}
    =
    1
    =
    z_{M,\cP} + z_{M,\cP_Z'}
  \end{equation*}
  of \Cref{le:pp'}\Cref{item:le:pp':z} coincide.   
\end{theorem}
\begin{proof}
  The proof plan follows that of \Cref{th:ppcomplid}, with appropriate modifications.

  \begin{enumerate}[(I), wide=0pt]
  \item\label{item:th:strongpp:I}{\bf: $\cP$ and its $Z$-strong negation $\cP_Z'$ are equivalent for $N$ and $\cL_N(X_E)$.} Precisely as in part \Cref{item:pr:ppcomplid:I} of \Cref{th:ppcomplid}, with the $S$ of \Cref{eq:xopluss} finite this time around. 

    Sections \Cref{item:pr:ppcomplid:II} and \Cref{item:pr:ppcomplid:III} of the proof of \Cref{th:ppcomplid} now also apply to yield their respective analogues:
    
  \item\label{item:th:strongpp:II}{\bf: $\cP$ lifts from $N$ to $M$}, and 

  \item\label{item:th:strongpp:III}{\bf: The $Z$-strong negation $\cP_Z'$ lifts from $N$ to $M$.}

  \item\label{item:th:strongpp:IV}{\bf: Conclusion.} We now have a decomposition $1=z_{N,\cP}+z_{N,\cP_Z'}$ in $M$, and we will be done as soon as we argue that $z_{N,\cP_Z'}\in M$ is central.

    Restricting attention to the $\cP_Z'$ von Neumann algebra $z_{M,\cP'}M$ expecting onto its $W^*$-subalgebra $z_{M,\cP'}N$ (with the strong PP expectation induced \cite[\S 3.9]{bdh} by $E$), we can assume $M$ is $\cP_Z'$ to begin with. Now, $\cP_Z'$ lifts to $\cL_N(X_E)$ by \Cref{item:th:strongpp:III}, and then descends down to $N$ by \Cref{item:th:strongpp:I}. It follows that $1=z_{M,\cP_Z'}$ cannot be {\it strictly} larger than $z_{N,\cP_Z'}$, and we are done. 
  \end{enumerate}
\end{proof}

\pf{th:finqg-sametype}
\begin{th:finqg-sametype}
  Consider the {\it spectral projections}
  \begin{equation*}
    M\xrightarrow{\quad E_{\alpha}\quad} M_{\alpha}
    ,\quad
    \alpha\in \Irr \bG
  \end{equation*}
  onto the respective {\it isotypic subspaces} (\cite[Theorem 1.5]{podl_symm} or \cite[\S 2.2]{dcy}). In particular, we have a conditional expectation
  \begin{equation*}
    M\xrightarrow{\quad E:=E_{\cat{triv}}\quad} M_{\cat{triv}}=M^{\bG}.
  \end{equation*}
  The conclusion is a direct application of \Cref{cor:typres} for all but finiteness and \cite[Proposition 2.1]{fk_fin-ind} for the latter, as soon as we verify that $E$ is PP. Now, \cite[Lemma 2.5]{dcy} proves a ``local'' version of the desired result, valid for arbitrary compact quantum $\bG$: for every $\alpha\in\Irr \bG$ there is some $\lambda_{\alpha}>0$ with
  \begin{equation}\label{eq:dcyeea}
    E(x^*x)\ge \lambda_{\alpha} E_{\alpha}(x)^*E_{\alpha}(x)\ \text{matricially}
    ,\quad \forall x\in M.
  \end{equation}
  Assuming (as we are) that $\bG$ is finite and hence $\left|\Irr(\bG)\right|<\infty$, for every $x\in M$ we have  
  \begin{equation*}
    \begin{aligned}
      x^*x &= \left(\sum_{\alpha} E_{\alpha}(x)^*\right)
             \left(\sum_{\alpha} E_{\alpha}(x)\right)\\
           &\le |\Irr(\bG)|^2 \sum_{\alpha} E_{\alpha}(x)^* E_{\alpha}(x)
             \quad\text{by {\it Cauchy Schwartz} \cite[\S I.1.1.2]{blk}}\\
           &\le |\Irr(\bG)|^2\left(\max_{\alpha}\frac 1{\lambda_{\alpha}}\right) E(x^*x)
             \quad\text{by \Cref{eq:dcyeea}},
    \end{aligned}
  \end{equation*}
  all valid matricially. \Cref{eq:ppexp} thus holds for
  \begin{equation*}
    \lambda := \frac{1}{|\Irr\bG|^2}\min_{\alpha}\lambda_{\alpha},
  \end{equation*}
  and we are done.   
\end{th:finqg-sametype}

\addcontentsline{toc}{section}{References}

\begin{thebibliography}{10}

\bibitem{bdh}
Michel Baillet, Yves Denizeau, and Jean-Fran{\c{c}}ois Havet.
\newblock Indice d'une esperance conditionnelle. ({Index} of a conditional
  expectation).
\newblock {\em Compos. Math.}, 66(2):199--236, 1988.

\bibitem{blk}
B.~Blackadar.
\newblock {\em Operator algebras}, volume 122 of {\em Encyclopaedia of
  Mathematical Sciences}.
\newblock Springer-Verlag, Berlin, 2006.
\newblock Theory of $C^*$-algebras and von Neumann algebras, Operator Algebras
  and Non-commutative Geometry, III.

\bibitem{bly_latt}
T.~S. Blyth.
\newblock {\em Lattices and ordered algebraic structures}.
\newblock Universitext. London: Springer, 2005.

\bibitem{dcy}
Kenny De~Commer and Makoto Yamashita.
\newblock A construction of finite index {{\(C^*\)}}-algebra inclusions from
  free actions of compact quantum groups.
\newblock {\em Publ. Res. Inst. Math. Sci.}, 49(4):709--735, 2013.

\bibitem{dh-1}
Yves Denizeau and Jean-Fran{\c{c}}ois Havet.
\newblock Correspondences with finite index. {I}: {The} index of a vector.
\newblock {\em J. Oper. Theory}, 32(1):111--156, 1994.

\bibitem{dixw}
Jacques Dixmier.
\newblock {\em von {N}eumann algebras}, volume~27 of {\em North-Holland
  Mathematical Library}.
\newblock North-Holland Publishing Co., Amsterdam-New York, 1981.
\newblock With a preface by E. C. Lance, Translated from the second French
  edition by F. Jellett.

\bibitem{zbMATH03532021}
G.~A. Elliott and E.~J. Woods.
\newblock The equivalence of various definitions for a properly infinite von
  {Neumann} algebra to be approximately finite dimensional.
\newblock {\em Proc. Am. Math. Soc.}, 60:175--178, 1977.

\bibitem{ell_afd-2}
George~A. Elliott.
\newblock On approximately finite-dimensional von {Neuman} algebras. {II}.
\newblock {\em Can. Math. Bull.}, 21:415--418, 1978.

\bibitem{fk_fin-ind}
Michael Frank and Eberhard Kirchberg.
\newblock On conditional expectations of finite index.
\newblock {\em J. Oper. Theory}, 40(1):87--111, 1998.

\bibitem{zbMATH03940006}
Uffe Haagerup.
\newblock A new proof of the equivalence of injectivity and hyperfiniteness for
  factors on a separable {Hilbert} space.
\newblock {\em J. Funct. Anal.}, 62:160--201, 1985.

\bibitem{zbMATH00059991}
Paul Jolissaint.
\newblock Conditional expectation index and finite von {Neumann} algebras.
\newblock {\em Math. Scand.}, 68(2):221--246, 1991.

\bibitem{jones_fin-act-hyp}
Vaughan F.~R. Jones.
\newblock {\em Actions of finite groups on the hyperfinite type {{\(II_ 1\)}}
  factor}, volume 237 of {\em Mem. Am. Math. Soc.}
\newblock Providence, RI: American Mathematical Society (AMS), 1980.

\bibitem{NeTu13}
Sergey Neshveyev and Lars Tuset.
\newblock {\em Compact quantum groups and their representation categories},
  volume~20 of {\em Cours Sp\'ecialis\'es [Specialized Courses]}.
\newblock Soci\'et\'e Math\'ematique de France, Paris, 2013.

\bibitem{pw}
Brian Parshall and Jian~Pan Wang.
\newblock Quantum linear groups.
\newblock {\em Mem. Amer. Math. Soc.}, 89(439):vi+157, 1991.

\bibitem{pasch_wast-mod}
William~L. Paschke.
\newblock Inner product modules over {B}{{\(^*\)}}-algebras.
\newblock {\em Trans. Am. Math. Soc.}, 182:443--468, 1973.

\bibitem{podl_symm}
Piotr Podle\'{s}.
\newblock Symmetries of quantum spaces. {S}ubgroups and quotient spaces of
  quantum {${\rm SU}(2)$} and {${\rm SO}(3)$} groups.
\newblock {\em Comm. Math. Phys.}, 170(1):1--20, 1995.

\bibitem{popa_clsf-subf}
Sorin Popa.
\newblock {\em Classification of subfactors and their endomorphisms}, volume~86
  of {\em Reg. Conf. Ser. Math.}
\newblock Providence, RI: AMS, American Mathematical Society, 1995.

\bibitem{pw_vna-act}
Sorin Popa and Antony Wassermann.
\newblock Actions of compact {Lie} groups on von {Neumann} algebras.
\newblock {\em C. R. Acad. Sci., Paris, S{\'e}r. I}, 315(4):421--426, 1992.

\bibitem{zbMATH01544901}
Michael Skeide.
\newblock Generalised matrix {{\(C^*\)}}-algebras and representations of
  {Hilbert} modules.
\newblock {\em Math. Proc. R. Ir. Acad.}, 100A(1):11--38, 2000.

\bibitem{strat}
\c{S}erban Str\u{a}til\u{a}.
\newblock {\em Modular theory in operator algebras}.
\newblock Editura Academiei Republicii Socialiste Rom\^{a}nia, Bucharest;
  Abacus Press, Tunbridge Wells, 1981.
\newblock Translated from the Romanian by the author.

\bibitem{tak1}
M.~Takesaki.
\newblock {\em Theory of operator algebras. {I}}, volume 124 of {\em
  Encyclopaedia of Mathematical Sciences}.
\newblock Springer-Verlag, Berlin, 2002.
\newblock Reprint of the first (1979) edition, Operator Algebras and
  Non-commutative Geometry, 5.

\bibitem{tak3}
M.~Takesaki.
\newblock {\em Theory of operator algebras. {III}}, volume 127 of {\em
  Encyclopaedia of Mathematical Sciences}.
\newblock Springer-Verlag, Berlin, 2003.
\newblock Operator Algebras and Non-commutative Geometry, 8.

\bibitem{wan-erg}
Shuzhou Wang.
\newblock Ergodic actions of universal quantum groups on operator algebras.
\newblock {\em Comm. Math. Phys.}, 203(2):481--498, 1999.

\end{thebibliography}

\def\polhk#1{\setbox0=\hbox{#1}{\ooalign{\hidewidth
  \lower1.5ex\hbox{`}\hidewidth\crcr\unhbox0}}}

\Addresses

\end{document}